\documentclass[reqno]{amsart}
\usepackage{verbatim}
\usepackage{amssymb}
\usepackage{enumerate}

\usepackage[mathletters]{ucs}
\numberwithin{equation}{section}
\usepackage{skull}

\usepackage[utf8x]{inputenc}
\usepackage{tikz}
\usetikzlibrary{decorations,decorations.markings}

\tikzset{
  strike through/.style={
    postaction=decorate,
    decoration={
      markings,
      mark=at position 0.5 with {
        \draw[-] (-3pt,-3pt) -- (3pt, 3pt);
      }
    }
  }
}

\newtheorem{theorem}{Theorem}[section]
\newtheorem{proposition}[theorem]{Proposition}
\newtheorem{lemma}[theorem]{Lemma}
\newtheorem{corollary}[theorem]{Corollary}

\newtheorem{problem}{Problem}

\newtheorem{example}[theorem]{Example}

\theoremstyle{definition}

\newtheorem{definition}[theorem]{Definition}

\newcommand{\mbb}[1]{\mathbb{#1}}

\newcommand{\setm}{\setminus}
\newcommand{\empt}{\emptyset}
\newcommand{\subs}{\subset}
\newcommand{\dom}{\operatorname{dom}}
\newcommand{\ran}{\operatorname{ran}}
\newcommand{\cof}{\operatorname{cof}}
\newcommand{\cov}{\operatorname{cov}}
\def\<{\left\langle}
\def\>{\right\rangle}
\def\br#1;#2;{\bigl[ {#1} \bigr]^ {#2} }

\newcommand{\cl}{\operatorname{cl}}

\newcommand{\snode}[2]{\node (#1) [ pro] {\tiny #2};}

\author[I. Juh\'asz]{Istv\'an Juh\'asz}
\address
      { Alfréd Rényi Institute of Mathematics, Hungarian Academy of Sciences
}
\email{juhasz@renyi.hu}

\author[L. Soukup]{Lajos Soukup}
\address
      { Alfréd Rényi Institute of Mathematics, Hungarian Academy of Sciences}
\email{soukup@renyi.hu}
\urladdr{http://www.renyi.hu/$\tilde{}$soukup}

\author[Z. Szentmikl\'ossy]{Zolt\'an Szentmikl\'ossy}
\address{E\"otv\"os University of Budapest}
\email{szentmiklossyz@gmail.com
}

\title[Between countably compact and $\omega$-bounded]
{Between countably compact and $\omega$-bounded}
\thanks{The research on and preparation of this paper was
supported by  OTKA grant no. K 83726.}

\date{\today}

\dedicatory{{Dedicated to the memory of Mary Ellen Rudin}}

\begin{document}

\begin{abstract}
Given a property $P$ of subspaces of a $T_1$ space $X$, we say that $X$ is
{\em $P$-bounded} iff every subspace of $X$ with property $P$ has compact closure
in $X$. Here we study $P$-bounded spaces for the properties $P \in \{\omega D, \omega N, C_2 \}$
where $\omega D \, \equiv$ ``countable discrete",  $\omega N \, \equiv$ ``countable nowhere dense",
and $C_2 \,\equiv$ ``second countable". Clearly, for each of these
$P$-bounded is between countably compact and $\omega$-bounded.

We give examples in ZFC that separate all these boundedness properties and their appropriate combinations.
Consistent separating examples with better properties (such as: smaller cardinality or weight, local
compactness, first countability) are also produced.

We have interesting results concerning $\omega D$-bounded spaces
which show that $\omega D$-boundedness is much stronger than
countable compactness:

\begin{itemize}
\item  Regular $\omega D$-bounded spaces of Lindelöf degree $< \cov(\mathcal{M})$ are $\omega$-bounded.

\item  Regular $\omega D$-bounded spaces of countable tightness are $\omega N$-bounded, and if
$\mathfrak{b} > \omega_1$ then even $\omega$-bounded.

\item  If a product of Hausdorff space is $\omega D$-bounded then all but one of its factors must be $\omega$-bounded.

\item  Any product of at most $\mathfrak{t}$ many Hausdorff $\omega D$-bounded spaces is countably compact.
\end{itemize}

As a byproduct we obtain that regular, countably tight, and countably compact spaces are discretely
generated.

\end{abstract}

\maketitle

\section{Introduction}

The work we report on in this paper can be considered as complementary to that in \cite{JvMW} where
certain natural boundedness properties that are strengthenings of the $\omega$-bounded
property had been investigated. Here we go in the opposite direction in the sense that
we study boundedness properties that are weaker than $\omega$-bounded but still stronger
than countably compact.

\begin{definition}
Given a property $P$ of subspaces of a space $X$, we say that $X$ is
{\em $P$-bounded} iff every subspace of $X$ with property $P$ has compact closure
in $X$.
\end{definition}

All spaces considered in this paper are assumed to be $T_1$. We intend to study
$P$-bounded spaces for the following choices of a property $P$ of subspaces:

\bigskip

\begin{itemize}
\item $P \,\equiv\,$``countable discrete" (in short: $\omega D$)

\medskip

\item $P \,\equiv\,$ ``countable nowhere dense" (in short: $\omega N$)

\medskip

\item $P \,\equiv\,$ ``second countable" (in short: $C_2$)

\end{itemize}

\bigskip

Let us note that, as second countable spaces are separable, we clearly have

\begin{center}
$C_2$-bounded $\,\equiv\, \omega C_2$-bounded $\,\equiv\, \omega C_1$-bounded,
\end{center}

\noindent where,
of course, $C_1$ abbreviates "first countable".

It is obvious that $\omega D$-bounded spaces are countably compact. Moreover,
as discrete subspaces of a crowded space are nowhere dense, it follows that
both $\omega N$-bounded and $C_2$-bounded crowded spaces are $\omega D$-bounded.
The following diagram gives a visual summary of these easy implications
between our boundedness properties and their natural combinations.

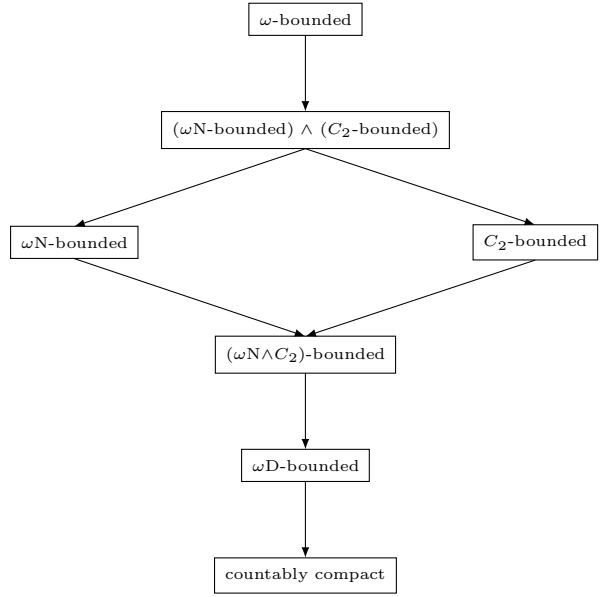
\begin{figure}[h]
\centering
\begin{tikzpicture}
[pro/.style={rectangle, outer sep=0mm, inner sep=4pt,minimum size=2mm,draw=black},
]
\matrix[row sep=10mm,column sep=3mm]{

&\snode{bounded}{${\omega}$-bounded};\\

&\snode{nwdC}{(${\omega}$N-bounded) $\land $ ($C_2$-bounded)};\\

\snode{nwdbounded}{${\omega}$N-bounded};& ;&;
\snode{m2bounded}{$C_2$-bounded};\\

&\snode{m2nwdbounded}{(${\omega}$N$\land$$C_2$)-bounded};\\

&\snode{Dbounded}{${\omega}$D-bounded};\\

&\snode{ccompact}{countably compact};&;
\\
};

\draw[>=latex,->] (bounded.south) --  (nwdC.north);

\draw[>=latex,->] (nwdC.south) --  (nwdbounded.north);

\draw[>=latex,->] (nwdC.south) --  (m2bounded.north);

\draw[>=latex,->] (nwdbounded.south) --  (m2nwdbounded.north);

\draw[>=latex,->] (m2bounded.south)  --  (m2nwdbounded.north);

\draw[>=latex,->] (m2nwdbounded.south) -- (Dbounded.north);
\draw[>=latex,->] (Dbounded.south) --  ( ccompact.north);

\end{tikzpicture}
\caption{ZFC implications}
\end{figure}

\section{Separation in ZFC}

The aim of this section is to show that all the seven classes of crowded
spaces appearing in this diagram are distinct. In fact, we are going to produce, in ZFC, zero-dimensional
crowded spaces that separate these classes.
Before we turn to constructing these examples, however, we need to do some preliminary work.

\begin{definition}
Given a property $P$ of subspaces of a space $X$ and a fixed subset $Y \subset X$, we write
$$[Y]^P = \{A \subs Y : A \mbox{ has property } P \},$$

and

$$\cl^P_X(Y) = \bigcup \{\overline{A} : A \in  [Y]^P \}.$$
Also, we say that $Y$ is {\em $P$-closed} in $X$ if $Y = \cl^P_X(Y)$.
(As usual, we shall omit the subscript $X$ when it is clear from the context.)
\end{definition}

Obviously, if $X$ is compact then every $P$-closed subspace of $X$ is $P$-bounded,
and if $X$ is also Hausdorff then the converse is also true.

Note that, in general, the subspace $\cl^P_X(Y)$ will not be $P$-closed in $X$.
But, at least for regular spaces, this will happen if $P = \omega D$ or $P = \omega N$.
In the first case this is easy and well-known: it follows from the fact that in a
regular space the points of any countable discrete set can be separated by disjoint
open sets. In the second case, when $P = \omega N$, we think that this observation is new.

\begin{lemma}\label{wN-closed}
Assume that $X$ is regular and $Y \subs X$. Then for any $H \in [X]^{\omega N}$ with
$H \subs \cl^{\omega N}(Y)$ there is $K \in [Y]^{\omega N}$ such that $H \subs \overline{K}$.
Consequently, $\cl^{\omega N}(Y)$ is $\omega N$-closed in $X$.
\end{lemma}

\begin{proof}
Let us fix an $\omega$-type enumeration $H = \{h_n : n < \omega\}$ of $H$ and for each
$n < \omega$ pick $K_n \in [Y]^{\omega N}$ such that $h_n \in \overline{K_n}$.
Then take a maximal disjoint collection $\mathcal{U} \subs \tau(X)$ such that
$\overline{U} \cap H = \empt$ holds true
for all $U \in \mathcal{U}$. Clearly, $\cup \mathcal{U}$ is dense open in $X$
because $H$ is nowhere dense.

Since $\bigcup \{K_n : n < \omega \}$ is countable,
so is $\mathcal{V} = \{U \in \mathcal{U} : U \cap \bigcup \{K_n : n < \omega \} \ne \empt\}$,
hence we may write $\mathcal{V} = \{U_m : m < \omega\}$. Let us now put
$$K = \bigcup_{n < \omega} \big(K_n \setm \cup_{m < n} U_m \big).$$
Then for each $m < \omega$ we have $U_m \cap K \subs \cup_{i \le m} K_i$, hence $U \cap K$
is nowhere dense for each $U \in \mathcal{U}$. As $X \setm \cup \mathcal{U}$ is nowhere dense,
this implies that $K \in [Y]^{\omega N}$.

Finally, for every $n < \omega$ we have $h_n\notin \bigcup_{m < n}\overline{U_m}$, so
$h_n\in \overline{(K_n\setm \bigcup_{m\le n} U_m)} \subs \overline{K}$, consequently
$H\subs \overline{K}$,  completing the proof.
\end{proof}

Our next preliminary step concerns a general construction that has as its input
a Tychonov space $X$ and a (Hausdorff) compactification $cX$ of $X$. Its output is a Tychonov space $M(X,cX)$
that is locally compact if so is $X$ and zero-dimensional if so is $cX$.

Let $\mathbb{C}$ denote the Cantor set and let $Y$ be the one-point compactification of the
locally compact and $\omega$-bounded product space $\omega_1 \times \mathbb{C}$. So $Y$ is what we may call the
``long Cantor set".
The compactifying point at infinity in $Y$ will be denoted by $y$.

\begin{definition}\label{M}
Let $X,\,cX,\,Y$ and $y \in Y$ be as above. Then we put $$M(X,cX) = \big[ cX\,\times\,(Y \setm \{y\})\big] \cup (X\,\times\,\{y\}),$$
considered as a subspace of $cX \times Y$.
\end{definition}

Note that $M(X,cX)$ is crowded, irrespective of whether $X$ is or not.
Note also that $X\,\times\,\{y\}$ is a closed and nowhere dense homeomorphic copy of $X$ in $M(X,cX)$,
moreover, $cX\,\times\,(Y \setm \{y\})$ is clearly $\omega$-bounded. The following proposition is an
immediate consequence of these two facts.

\begin{proposition}\label{M-appl}
Let $P$ be any of the properties appearing in the diagram of Figure 1.
\begin{enumerate}[(1)]
\item If $X$ is $P$-bounded then so is $M(X,cX)$.

\smallskip

\item If $X$ is not $P$-bounded then $M(X,cX)$ is not $(P \wedge \omega N)$-bounded.
In particular, if $X$ is not $\omega$-bounded then $M(X,cX)$ is not $\omega N$-bounded.
\end{enumerate}
\end{proposition}

We shall also make use of the following simple preparatory lemma.

\begin{lemma}\label{chi}
Let $X$ be a regular space all whose points have uncountable character.
Then every first countable subspace $Y$ of $X$ is nowhere dense in $X$.
\end{lemma}

\begin{proof}
Assume, on the contrary, that $U \subs X$ is non-empty open and $Y \cap U$
is dense in $U$. But then, see e.g. 2.7 of \cite{J1}, for each point $y \in Y \cap U$
we have $\chi(y,Y) = \chi(y,U) = \chi(y,X) > \omega$, a contradiction.
\end{proof}

We are now ready to produce (in ZFC) the examples that separate the classes from our diagram.

\begin{example}\label{cc-notwD}
Let $K$ be a zero-dimensional, crowded compactum in which all infinite closed sets have size $> \mathfrak{c}$.
(The \v Cech-Stone remainder $\omega^*$ is such.) Then $K$ has a countable crowded subset $S$
and there is a countably compact subspace $X$ such that $S \subs X \subs \overline{S}$ and
$|X| = \mathfrak{c}$. Thus, $X$ is a countably compact, zero-dimensional, crowded space in which
all compact subsets are finite, hence it is clearly not $\omega D$-bounded.
\end{example}

Our next example will be $\omega N$-bounded but not $C_2$-bounded. It will be a subspace of
$\beta \mathbb{Q}$, the \v Cech-Stone compactification of the rationals.

\begin{example}\label{wN-notC2}
Let us put $X = cl^{\omega N}_{\beta \mathbb{Q}}\big ( \mathbb{Q}\big )$, the $\omega N$-closure
of $\mathbb{Q}$ in $\beta \mathbb{Q}$. Then, by lemma \ref{wN-closed}, $X$ is $\omega N$-closed
in $\beta \mathbb{Q}$, hence $X$ is $\omega N$-bounded.

On the other hand, by theorem 1.5 of \cite{vD} $\,\mathbb{Q}$ has a remote point $p$ because
it is not pseudocompact and has countable ($\pi$-)weight. But then $p \notin X$, hence
$X$ is not compact while the second countable $\mathbb{Q}$ is dense in $X$.
So, $X$ is not $C_2$-bounded.
\end{example}

We can now apply our ``M-machine" from \ref{M} to the space
$X = cl^{\omega N}_{\beta \mathbb{Q}}\big ( \mathbb{Q}\big )$ and its compactification $\beta \mathbb{Q}$.

\begin{example}\label{wD-notwN&C2}
The space $M(cl^{\omega N}_{\beta \mathbb{Q}}\big ( \mathbb{Q}\big ),\beta \mathbb{Q}\,)$
is $\omega D$-bounded by part (1) of \ref{M-appl} and it is
not $(\omega N \wedge C_2)$-bounded by part (2) of \ref{M-appl}.
\end{example}

Our aim next is to produce a space that is both $\omega N$-bounded and $C_2$-bounded but
not $\omega$-bounded. To do that, we first refer to \cite{d-forced} where it was shown
that there is a countable dense subspace $S$ of the Cantor cube $2^\mathfrak{c}$ which is
{\em nodec}, that is every nowhere dense subset of $S$ is closed in $S$.
(In fact, the space $S$ constructed there has a stronger property: it is submaximal.) Let us pick a member
$a \in S$ and consider $X = \cl^{\omega N}_{2^\mathfrak{c}}(S \setm \{a\})$, the $\omega N$-closure
of $S \setm \{a\}$ in $2^\mathfrak{c}$.

Applying lemma \ref{wN-closed}, we can see that $X$ is $\omega N$-closed in $2^\mathfrak{c}$,
and hence it is $\omega N$-bounded. Since $S$ is nodec, it follows that $a \notin \overline{T}$
for every (countable) nowhere dense subset of $S \setm \{a\}$, hence $a \notin X$. But $S \setm \{a\}$
is still dense in $2^\mathfrak{c}$ as well as in $X$, consequently $X$ is not $\omega$-bounded
because it is not compact.

Note that, as $X$ is dense in $2^\mathfrak{c}$, for all points $x \in X$ we have $\chi(x,X) = \mathfrak{c}$.
Consequently, every (countable and) second countable subspace $T$ of $X$ is nowhere dense by lemma
\ref{chi}. But this implies that $\overline{T}$ is compact because $X$ is already known to be
$\omega N$-bounded, and thus we see that $X$ is also $C_2$-bounded.

Putting all this together we get what we wanted.

\begin{example}\label{wN-C2-notw}
The zero-dimensional space $X = \cl^{\omega N}_{2^\mathfrak{c}}(S \setm \{a\})$
is both $\omega N$-bounded and $C_2$-bounded but not $\omega$-bounded.
\end{example}

Since $2^\mathfrak{c}$ is a compactification of the space $X$ from our last example
\ref{wN-C2-notw}, we may apply our ``$M$-machine" to $X$ and $2^\mathfrak{c}$
similarly as we did in example \ref{wD-notwN&C2}. This, in view of
lemma \ref{M-appl}, will preserve the $C_2$-bounded property of $X$ and will turn
its lack of $\omega$-boundedness into the lack of $\omega N$-boundedness.

\begin{example}\label{C2-notwN}
 The zero-dimensional space $M(X,2^\mathfrak{c})$ is $C_2$-bounded but not $\omega N$-bounded.
\end{example}

Finally, we note that if $H$ is any space that is $\omega N$-bounded
but not $C_2$-bounded and, conversely, $K$ is $C_2$-bounded but not $\omega$-bounded, then their
topological sum $H \oplus K$ is clearly $(\omega N \wedge C_2)$-bounded but neither
$\omega N$-bounded nor $C_2$-bounded. Thus we have our last separating example.

\begin{example}\label{last}
The topological sum of the spaces from example \ref{wN-notC2} and from example \ref{C2-notwN}
is a zero-dimensional crowded space that is $(\omega N \wedge C_2)$-bounded but neither
$\omega N$-bounded nor $C_2$-bounded.
\end{example}

It is well-known that each crowded, countably compact, regular space has cardinality at least $\,\mathfrak{c}$.
However, all our above examples with the exception of \ref{cc-notwD} are of cardinality $2^\mathfrak{c}$.
We do not know if this is really necessary, hence we formulate the following problem.

\begin{problem}
Are there, in ZFC, separating examples like 2.7 -- 2.11 of cardinality less than $2^\mathfrak{c}$?
\end{problem}

Consistently, at least in some cases, we can produce such separating examples of size less than $2^\mathfrak{c}$.
This is the topic of the next section.

\section{Small and/or locally compact separating examples}

Let us note, first of all, that the well-known Franklin-Rajagopalan space
(usually denoted by $\gamma \mathbb{N}$,
see e.g. \cite{FrRa}) that is   
obtained from a tower $\{A_\alpha : \alpha < \mathfrak{t} \} \subs [\omega]^\omega$ is locally
compact, sequentially compact, hence countably compact, but not $\omega D$-compact. Of course, this space is not crowded
but by taking its product with, say, the Cantor set we get a similar crowded and zero-dimensional example.
Moreover, as $|\mathbb{C} \times \gamma \mathbb{N}| = \mathfrak{c}$, this example is
also small, just like the similar, but not locally compact, example in \ref{cc-notwD}.
It is worth mentioning that if $\mathfrak{t} = \omega_1$ then this example is also first countable.

A somewhat similar construction has recently been carried out in \cite{HG}, however not in ZFC but
using the extra assumption $\mathfrak{p} = \cof(\mathcal{M})$. This yields a locally compact Hausdorff
space $X$ on the underlying set $(\omega \times \mathbb{C}) \cup \mathfrak{p}$ such that
$(\omega \times \mathbb{C})$ in its natural topology is dense open in $X$, the closed
nowhere dense subspace $\mathfrak{p}$ carries its ordinal topology, hence $X$ is not compact,
moreover $X$ is $\omega N$-bounded. Since $(\omega \times \mathbb{C})$ is a dense $C_2$ subspace
of $X$ this yields us the following.

\begin{example}\label{mex}
If $\,\,\mathfrak{p} = \cof(\mathcal{M})$ then $X$ from \cite{HG} above is a zero-dimensional, crowded,
locally compact space of cardinality $\mathfrak{c}$
that is $\omega N$-bounded but not $C_2$-bounded. If, in addition, $\mathfrak{p} = \cof(M) = \omega_1$
then $X$ is also first countable.
\end{example}

Now, we may again apply our $M$-machine to $X$ and its one-point compactification $\alpha X$
to obtain the following:

\begin{example}\label{mex+}
If $\,\,\mathfrak{p} = \cof(\mathcal{M})$ then
$M(X, \alpha X)$ is a crowded, zero-dimensional, locally compact
space of cardinality $\mathfrak{c}$  
 that is $\omega D$-bounded but not $(C_2 \wedge \omega N)$-bounded.
\end{example}

\begin{definition}\label{touch}
Let $X$ be a compactum and $x \in X$ a non-isolated point. For $P$, a property of subspaces,
we say that $x$ is $P$-touchable iff there is 
$Y\in [X \setm \{x\}]^P$ 
such that $x \in \overline{Y}$.
\end{definition}

Consequently, the task to find a locally compact Hausdorff space that is $P$-bounded but not $Q$-bounded
is equivalent to the task of finding a $Q$-touchable but not $P$-touchable point in a compactum.
Luckily for us, the latter task has already been dealt with, mainly by J. van Mill.

In particular, he proved in \cite{vM82} that in the \v Cech-Stone remainder $\omega^*$ there is a
point that is $\omega D$-untouchable but $\omega$-touchable ($\equiv$ not a weak P) point.
Since $\omega^*$ has no non-trivial convergent sequences, it is easy to see that every $C_2$ subspace of $\omega^*$
is $\omega D$, hence we have the following.

\begin{example}\label{vM}
If $p \in \omega^*$ is as above then $\omega^* \setm \{p\}$ is a locally compact space
that is $C_2$-bounded but not $\omega$-bounded, in fact not even $\omega N$-bounded.
\end{example}

We do not know if there is a ZFC example of a locally compact space that is both
$\omega N$-bounded and $C_2$-bounded but not $\omega$-bounded. However, such an example
exists if we assume $\mathfrak{b} = \mathfrak{c}$. Indeed, by theorem 9.1 of \cite{vM1},
$\mathfrak{b} = \mathfrak{c}$ implies that if $X$ is any non-pseudocompact Tychonov space
then there is a point in $x \in X^* = \beta X \setm X$ that is $\omega N$-untouchable in $\beta X$.
Thus, if $X$ is also separable, then $\beta X \setm \{x\}$ is $\omega N$-bounded and
not $\omega$-bounded.

\begin{example}\label{vM-b=c}
Assume $\mathfrak{b} = \mathfrak{c}$ and
apply the above argument to the separable, non-pseudocompact space $X = \omega \times 2^{\omega_1}$
to obtain $x \in X^*$ such that $\beta X \setm \{x\}$ is $\omega N$-bounded
and not $\omega$-bounded.
But now all points
of $\beta X$ have uncountable character, hence, by lemma \ref{chi}, every $C_2$ subspace of it is nowhere dense.
Consequently,  $\beta X \setm \{x\}$ is also $C_2$-bounded.
\end{example}

Similarly as we obtained example \ref{C2-notwN} from example \ref{wN-C2-notw} by running the latter
through our "$M$-machine", we get from example \ref{vM-b=c} the following.

\begin{example}\label{vM-b=c_2}
Assume $\mathfrak{b} = \mathfrak{c}$. Then for $X = \omega \times 2^{\omega_1}$ and $x \in X^*$
as in example \ref{vM-b=c}, the locally compact space $M(\beta X \setm \{x\}, \beta X\,)$
is $C_2$-bounded but not $\omega N$ bounded.
\end{example}

Now, if the continuum hypothesis holds, then both $\mathfrak{b} = \mathfrak{c}$
and $\mathfrak{p} = \cof(\mathcal{M}) = \omega_1$ are valid, hence we get the following.

\begin{example}\label{CH}
If the continuum hypothesis holds then the topological sum of the spaces in \ref{vM-b=c_2} and \ref{mex} is
a locally compact, crowded, zero-dimensional space that is $(\omega N \wedge C_2)$-bounded, but it is neither
$\omega N$-bounded nor $C_2$-bounded.
\end{example}

We end this section with a result that yields consistent small
examples similar to examples
\ref{wN-C2-notw} and \ref{C2-notwN}. The simple idea is to find  $\lambda < \mathfrak{c}$
such that the Cantor cube $\mathbb{C}_\lambda$ of weight $\lambda$ has a countable, dense,
nodec subspace. Then these examples can be produced in exactly the same way. (We shall use the
more general notation $\mathbb{C}_I$ instead of $2^I$.) So, for instance, $2^{\omega_1}$ will
be used to denote the cardinality of $\mathbb{C}_{\omega_1}$.

Let us note, to begin with, that if $\lambda < \mathfrak{p}$ then every countable, dense
subspace of $\mathbb{C}_\lambda$ is Fr\`echet, hence cannot be nodec. Consequently, the above plan
can not be carried out in ZFC, so we only may hope for a consistency result.

In what follows, we shall denote by $\mathcal{N}(X)$ the ideal of all
nowhere dense subsets of the space $X$. Moreover, the cofinality of the ideal $\mathcal{N}(\mathbb{C}_\lambda)$
will be denoted by $\nu_\lambda$.

\begin{theorem}\label{cof(M)=w_1}
If  $\,\lambda = \nu_\lambda < \mathfrak{c}$ then the Cantor cube $\mathbb{C}_\lambda$
has a countable dense subspace that is {\em nodec}.
\end{theorem}

\begin{proof}
Let us start with what we know: a countable, dense, nodec subspace $S$ of $\mathbb{C}_\mathfrak{c}$.
We shall prove that there is a set $I \in [\mathfrak{c}]^\lambda$ such that the projection
$\pi_I[S]$ of $S$ to this set of co-ordinates is nodec. Since $\pi_I[S]$ is clearly countable and
dense in $\mathbb{C}_I$, this will complete our proof.

As each $T \in \mathcal{N}(S)$ is discrete, we may fix for $T$ a countable set of co-ordinates
$E(T) \in [\mathfrak{c}]^\omega$ such that the projection $\pi_{E(T)}$ is one-to-one on $T$
and the image $\pi_{E(T)}[T]$ is discrete in $\mathbb{C}_{E(T)}$. Note that then $\pi_{J}[T]$ is
also discrete in $\mathbb{C}_J$ whenever $E(T) \subs J \subs \mathfrak{c}$.

Also, it follows from $\nu_\lambda = \lambda$ that for every set of co-ordinates
$J \in [\mathfrak{c}]^\lambda$ of size $\lambda$ we can fix
a collection $\mathcal{H}_J \in [\mathcal{N}(\mathbb{C}_J)]^\lambda$ that is cofinal in
$\mathcal{N}(\mathbb{C}_J)$.

For any set $J \subs \mathfrak{c}$ and partial function $y \in \mathbb{C}_J$ we write
$[y] = \{x \in \mathbb{C}_\mathfrak{c} : y \subs x\}$. Moreover, for any subset $H \subs \mathbb{C}_J$
we put $[H] = \cup \{[y] : y \in H \}$. Clearly, if $H \subs \mathbb{C}_J$ is nowhere dense in
$\mathbb{C}_J$ then $[H]$ is nowhere dense in $\mathbb{C}_\mathfrak{c}$. This implies that the
following operation $J \mapsto J^+$ on  $[\mathfrak{c}]^\lambda$ is well-defined.

For any $J \in [\mathfrak{c}]^\lambda$ we set $$J^+ = J \cup \bigcup \{E(S \cap [H]) : H \in \mathcal{H}_J \}\,.$$

We now define sets $I_\xi \in [\mathfrak{c}]^\lambda$ by transfinite recursion for $\xi < \omega_1$
as follows. $I_0$ is an arbitrary member of $[\mathfrak{c}]^\lambda$. If $I_\xi \in [\mathfrak{c}]^\lambda$
has been defined then $I_{\xi+1} = (I_\xi)^+$. Finally, if $\xi < \omega_1$ is limit then we put
$I_\xi = \cup_{\eta < \xi} I_\eta$.

We claim that $I = \cup_{\xi < \omega_1} I_\xi \in [\mathfrak{c}]^\lambda$ is as required, i.e.
$\pi_I[S]$ is nodec. To see this, note first that those members of $\mathcal{N}(\mathbb{C}_I)$
that depend only on a countable number of co-ordinates are cofinal in $\mathcal{N}(\mathbb{C}_I)$.
This is immediate from the fact any dense open subset of $\mathbb{C}_I$ includes a dense open set
that is the countable union of elementary open sets, for $\mathbb{C}_I$ is CCC.

But now, if $Z \in \mathcal{N}(\mathbb{C}_I)$ depends only on the countable set
of co-ordinates $A \subs I$ then
there is some $\xi < \omega_1$ with $A \subs I_\xi$. Clearly, then the projection $Y$ of $Z$
to $\mathbb{C}_{I_\xi}$ is nowhere dense there. Consequently, there is $H \in \mathcal{H}_{I_\xi}$
with $Y \subs H$. By $A \subs I_\xi$ this implies $[Z] \subs [H]$.
Now, by definition, we have $E(S \cap [H]) \subs I_{\xi+1} \subs I$ which implies that the projection
$\pi_I \big [S \cap [H] \big]$ is discrete, hence so is its subset $Z \cap \pi_I[S]$.

This immediately implies that all members of $\mathcal{N}(\pi_I[S])$ are discrete, and hence
also closed-and-discrete. Consequently, $\pi_I[S]$ is indeed nodec.

\end{proof}

It is known that if one adds Sacks reals side-by-side to a model of GCH then in the resulting
generic extension we have $\mathfrak{c}=2^{{\omega}_1}$. Moreover,
in this model $\cof(\mathcal{M}) = \omega_1$ hols, see corollary 66 in \cite{GQ}.
But by 1.5 and 1.6 of \cite{BaHEHr} we have $\cof(\mathcal{M}) = \nu_\omega$, 
consequently $\nu_\omega = \nu_{\omega_1} = \omega_1$
holds. This shows that the assumption of theorem \ref{cof(M)=w_1} is meaningful.

The same arguments that we used to obtain examples \ref{wN-C2-notw} and \ref{C2-notwN} now yields
the following results.

\begin{corollary}
If  $\,\lambda = \nu_\lambda < \mathfrak{c}$ then the Cantor cube $\mathbb{C}_\lambda$
has a separable, dense subspace $X$ that is both $\omega N$-bounded and $C_2$-bounded but not
compact, hence not $\omega$-bounded.  Moreover, then $M(X,\mathbb{C}_\lambda)$ is $C_2$-bounded but not
$\omega N$-bounded. In particular, in the generic extension obtained by
adding many Sacks reals side-by-side to a model of GCH, 
there are examples like these of size $\mathfrak{c}$.
\end{corollary}

To conclude this section we recall from \cite{J2} that the existence of an $HFD$ space of size
$\mathfrak{c}$ implies the existence of a countably compact $HFD$, which is, on one
hand hereditarily separable, and on the other hand it has no infinite compact subspaces.
(We know that there are many models of set theory
in which $HFD$ spaces of size $\mathfrak{c}$ exist.) Since a countably compact but not compact
space cannot be Lindelöf, such a hereditarily separable regular example is necessarily
an S-space, hence they cannot exist in ZFC.

\section{On  $\omega D$-bounded spaces}

In this section our aim is to show that, while the $\omega D$-bounded property is
quite low in the hierarchy of our boundedness properties, it is actually much stronger
than countable compactness.
The first evidence of this that we give is a strengthening of the well-known fact that
regular and countably compact spaces of Lindelöf degree $< \mathfrak{p}$ are
$\omega$-bounded.

\begin{theorem}\label{Lind}
Regular and $\omega D$-bounded spaces of Lindelöf degree $< \cov(\mathcal{M})$ are $\omega$-bounded.
\end{theorem}

\begin{proof}
Actually, we prove a slightly stronger result: If $X$ is a separable, regular space which admits an open cover
of size $< \cov(\mathcal{M})$ that has no finite subcover then $X$ is not $\omega D$-bounded.

To see this, fix a countable dense subset $S$ of $X$ and, using the regularity of $X$, an open cover $\mathcal{U}$
with $|\mathcal{U}|< \cov(\mathcal{M})$ such that for any finite $\mathcal{V} \subs \mathcal{U}$
we have $\overline{\cup \mathcal{V}} \ne X$. We then pick for each $x \in S$ a member $U_x \in \mathcal{U}$
such that $x \in U_x$.

Now, consider the sub partial order
$\mathbb{P} \subs S^{< \omega}$ consisting of those finite sequences $p \in S^{< \omega}$ that satisfy
$p(i) \notin \cup_{j < i} U_{p(j)}$ for all $i \in \dom(p)$. (Note that $\dom(p) \in \omega$.)

For every finite $\mathcal{V} \subs \mathcal{U}$ we define
$$E_{\mathcal{V}} = \{p \in \mathbb{P} : \exists\,i \in \dom(p)\, (p(i) \notin \overline{\cup \mathcal{V}})\}.$$
Then $E_{\mathcal{V}}$ is dense in $\mathbb{P}$ because, by our assumption, for every $p \in \mathbb{P}$
the closure of $\,\,\bigcup \{U_{p(i)} : i \in \dom(p) \} \cup \mathcal{V}$
does not cover $X$.

\smallskip

Now, $|\{E_\mathcal{V} : \mathcal{V} \in [\mathcal{U}]^{< \omega}\}| < \cov(\mathcal{M})$ implies that
there is a $\mathbb{P}$-generic set $G$ over the family $\{E_\mathcal{V} : \mathcal{V} \in [\mathcal{U}]^{< \omega}\}$.
Clearly, then $g = \cup G \in S^\omega$. We claim that $D = \ran(g) = \{g(i) : i \in \omega\}$ is discrete with non-compact
closure.

That $D$ is discrete follows because for every $i < \omega$ we have $$\{g(i)\} = D \cap U_{g(i)} \setm \{g(j) : j < i\}.$$
If $\overline{D}$ would be compact then we had a finite $\mathcal{V} \subs \mathcal{U}$ with $\overline{D} \subs \cup \mathcal{V}$.
This, however, would contradict $D \setm \cup \mathcal{V} \ne \emptyset$ that, in turn, follows from $G \cap E_{\mathcal{V}} \ne \emptyset$.

\end{proof}

One of the most notorious open problems (worth 1000 US Dollars) concerning countably compact spaces
asks if there is, in ZFC, a first countable, separable, regular, countably compact but non-compact space.
Equivalently, is it consistent that all first countable, regular, countably compact spaces are $\omega$-bounded?
The answer to this question is unknown even if we replace in it first countable with countably tight.
We refer the reader to \cite{Ny} concerning the history and the sate of this problem.

In what follows, we shall show that the situation changes drastically if in this question countably
compact is strengthened to $\omega D$-bounded.
We start by introducing some new notation. Let $X$ be any space, $S \subs X$ be a dense subset and $\mathcal{U}$
be a family of pairwise disjoint non-empty open sets in $X$. We shall denote by $\mathcal{I}(S,\mathcal{U})$
the collection of all those countable sets $D \subs S$ that satisfy $|D \cap U| < \omega$ for all $U \in \mathcal{U}$.
Clearly, $\mathcal{I}(S,\mathcal{U})$ is a subideal of $[S]^{\le \omega}$ and all members of $\mathcal{I}(S,\mathcal{U})$
are discrete subsets of $X$. We also recall that an ideal
$\mathcal{J}$ is a $P$-ideal iff for any $\{A_n : n < \omega \} \subs \mathcal{J}$
there is $A \in \mathcal{J}$ such that $A_n \subs^* A$, i.e. $A_n \setm A$ is finite, for all $n < \omega$.

\begin{lemma}\label{P}
Let $X,\,S$, and $\mathcal{U}$ be as above. Then $\mathcal{I}(S,\mathcal{U})$ is a $P$-ideal.
\end{lemma}

\begin{proof}
Let us consider any family $\{D_n : n < \omega\} \subs \mathcal{I}(S,\mathcal{U})$ and let
$\mathcal{V}$ be, the obviously countable, collection of all those elements of $\mathcal{U}$
that intersect $\cup \{D_n : n < \omega\}$. We may then write $\mathcal{V} = \{V_i : i < \omega\}$.
Then put $$D = \bigcup_{n < \omega} \big(D_n \setm \cup_{i< n} V_i \big).$$
We have $D \in \mathcal{I}(S,\mathcal{U})$ because $D \cap V_i \subs \cup_{n \le i}(D_n \cap V_i)$
is clearly finite for each $i < \omega$,
moreover $D \cap U = \emptyset$ whenever $U \in \mathcal{U} \setm \mathcal{V}$.
We also have that $D_n \setm D \subs \cup_{i < n}(D_n \cap V_i)$ is finite, hence $D_n \subs^* D$
for all $n < \omega$.

\end{proof}

\begin{lemma}\label{I}
Let $X,\,S$, and $\mathcal{U}$ be as above and assume, in addition, that $X$ is regular and
countably compact. Then the following are valid:

\smallskip

\begin{enumerate}[(i)]
\item $\overline{\cup \mathcal{U}} \setm \cup \{\overline{U} : U \in \mathcal{U}\} \subs \overline{\cup \{D' : D \in \mathcal{I}(S,\mathcal{U})\}}$;

\bigskip

\item if $x \in \overline{\cup \mathcal{U}} \setm \cup \{\overline{U} : U \in \mathcal{U}\}$ and
we also have $t(x,X) = \omega$ then even $x \in D'$
for some $D \in \mathcal{I}(S,\mathcal{U})$.

\end{enumerate}

\end{lemma}

\begin{proof}
(i) Let $x \in \overline{\cup \mathcal{U}} \setm \cup \{\overline{U} : U \in \mathcal{U}\}$ and
take any open set $V$ containing $x$. Then $V$ clearly intersects infinitely many members of $\mathcal{U}$,
hence there are distinct $\{U_n : n < \omega\} \subs \mathcal{U}$ such that $V \cap U_n \ne \emptyset$
for all $n < \omega$. Then we may choose $x_n \in S \cap V \cap U_n$ for each $n$, and clearly
$D = \{x_n : n < \omega\} \in \mathcal{I}(S,\mathcal{U})$.
The countable compactness of $X$ then implies $\emptyset \ne D' \subs \overline{V}$
which shows that every closed neighbourhood of $x$ intersects $\cup \{D' : D \in \mathcal{I}(S,\mathcal{U})\}$.
Thus by the regularity of $X$ we have established (i).

\medskip

(ii) Now if $x \in \overline{\cup \mathcal{U}} \setm \cup \{\overline{U} : U \in \mathcal{U}\}$ and
$t(x,X) = \omega$ then by part (i) there is a countable subset
$\{x_n : n < \omega\} \subs \cup \{D' : D \in \mathcal{I}(S,\mathcal{U})\}$
with $x \in \overline{\{x_n : n < \omega\}}$, hence $x \in \{x_n : n < \omega\}'$.

By definition, for each $n < \omega$ there is
some $D_n \in \mathcal{I}(S,\mathcal{U})$ such that $x_n \in D'_n$ and by lemma \ref{P} there
is $D \in \mathcal{I}(S,\mathcal{U})$ for which $D_n \subs^* D$
for all $n < \omega$. But, obviously, $D_n \subs^* D$ implies $D'_n \subs D'$,
consequently we have $x \in D'$.

\end{proof}

We are now prepared to formulate and prove the following, perhaps surprising, result.

\begin{theorem}\label{main}
Assume that the space $X$ is regular and countably compact and $S \subs X$ is dense in $X$.
Then for every countable nowhere dense set $A \subs X$ such that $t(x,X) = \omega$ for all
$x \in A$ there is a countable discrete subset $D$ of $S$ for which $A \subs \overline{D}$.
\end{theorem}

\begin{proof}
Let $\mathcal{U}$ be a maximal disjoint collection of non-empty open sets $U$ in $X$
such that $A \cap \overline{U} = \empt$. Then the regularity of $X$ and the maximality of $\mathcal{U}$
imply that $\cup \mathcal{U}$ is dense in $X$,
hence we have  $A \subs \overline{\cup \mathcal{U}} \setm \cup \{\overline{U} : U \in \mathcal{U}\}$. We may then apply
part (ii) of lemma \ref{I} to conclude that for each $x \in A$ there is some $D_x \in \mathcal{I}(S,\mathcal{U})$
such that $x \in D'_x$. By lemma \ref{P}, however, there is a set $D \in \mathcal{I}(S,\mathcal{U})$ for
which $D_x \subs^* D$ whenever $x \in A$, and then we obviously have $A \subs D' \subs \overline{D}$.

\end{proof}

The following ZFC "non-separation" result is then an immediate corollary of theorem \ref{main},
by simply putting $S = X$.

\begin{corollary}\label{wD=wN}
Any countably tight regular space that is $\omega D$-bounded is also $\omega N$-bounded.
\end{corollary}

In \cite{DTTW} it was proved that countably tight compacta are discretely determined, that
is if a point is in the closure of a set then it is in the closure of its discrete subset.
Our next corollary of theorem \ref{main} extends this result from compacta to regular and
countably compact spaces.

\begin{corollary}\label{DD}
All countably tight, regular, and countably compact spaces are discretely determined.
\end{corollary}

\begin{proof}
Clearly, this follows if we can show that for any countably tight, regular, and countably
compact space $X$, for any dense subset $S \subs X$, and for any non-isolated point $x \in X$
we have $x \in \cl^{\omega D}(S)$.
But for a non-isolated point $x \in X$ the singleton set $\{x\}$ is nowhere dense,
hence theorem \ref{main} can be applied to it and the dense set $S$.

\end{proof}

We can now move on to our promised result that for countably tight and regular spaces the
$\omega D$-bounded property implies $\omega$-boundedness, provided that $\mathfrak{b} > \omega_1$
holds.

To this end, we recall the well-known fact (see e.g. \cite{JS}, 1.7) that for the property ``$F \equiv$ free sequence",
being $F$-bounded is equivalent with compactness, that is if all free sequences in a space $X$ have compact
closures then $X$ is compact. This fact will be used in the proof of our result.

\begin{theorem}\label{b>w_1}
Assume that $\mathfrak{b} > \omega_1$ holds. Then every countably tight and regular space that is
$\omega D$-bounded is also $\omega$-bounded.
\end{theorem}

\begin{proof}
What we shall show is that every $\omega D$-bounded, countably tight, and regular {\em separable} space $X$
is compact. To do this, we first also assume that $X$ is crowded. We now fix a countable dense set $S \subs X$.

Let us assume, arguing indirectly, that $X$ is not compact.
According to the remark we made above, then there is a free sequence in $X$ that has non-compact closure.
Since free sequences are discrete and $X$ is $\omega D$-bounded, this free sequence cannot be countable, hence
there is a free sequence $F = \{x_\alpha : \alpha < \omega_1\}$ of type $\omega_1$ in $X$. The
countable tightness of $X$ then implies that $F$ cannot have a complete accumulation point in $X$,
hence $\overline{F}$ is not compact.

Now let us fix a maximal disjoint collection $\mathcal{U}$ of non-empty open sets $U$ in $X$
such that $F \cap \overline{U} = \empt$. Then $\cup \mathcal{U}$ is dense in $X$ because $F$,
being discrete, is nowhere dense in our crowded and regular space $X$. Then $|\mathcal{U}| = \omega$,
for $X$ is separable, and it cannot be finite as $\cup \{\overline{U} : U \in \mathcal{U}\} \ne X$.
Also, $|U \cap S| = \omega$ for each $U \in \mathcal{U}$ because $X$ is crowded.

Thus we may write $\mathcal{U} = \{U_n : n < \omega\}$ and $S \cap U_n = \{s_{n,i} : i <\omega\}$ for each $n$.
Note that for every set $D \in  \mathcal{I}(S,\mathcal{U})$ we can fix a function $f_D \in \omega^\omega$
such that $D \cap U_n \subs \{s_{n,i} : i < f_D(n)\})$ for every $n < \omega$.

It follows from part (ii) of lemma \ref{I} that for every $\alpha < \omega_1$ there is some
$D_\alpha \in  \mathcal{I}(S,\mathcal{U})$ with $x_\alpha \in D'_\alpha$. Now we apply $\mathfrak{b} > \omega_1$
by taking a function $f \in \omega^\omega$ such that $f_{D_\alpha} <^* f$ for all $\alpha < \omega_1$.
Let us then put $$D = \bigcup_{n < \omega} \{ s_{n,i} : i < f(n)\}.$$
Clearly, we have $D \in  \mathcal{I}(S,\mathcal{U})$ and $D_\alpha \subs^* D$ for each $\alpha < \omega_1$.
But this would imply $x_\alpha \in D'_\alpha \subs D' \subs \overline{D}$ for all $\alpha < \omega_1$,
hence $F \subs \overline{D}{}$ and this
is a contradiction because $\overline{D}$ is compact but $\overline{F}$ is not.

Now consider the general case in which $X$ is not necessarily crowded, hence $I(X)$, the set of the
isolated points of $X$, may be non-empty. Of course, $I(X)$ is countable as $X$ is separable.
Thus, if $I(X)$ is dense in $X$ then $X = \overline{I(X)}$ is compact. Otherwise, $G = X \setm \overline{I(X)}$
is an open, hence separable and crowded subspace, and so is its closure $\overline{G}$ which is $\omega D$-bounded.
Thus $\overline{G}$ is compact by the above and hence so is $X = \overline{I(X)} \cup \overline{G}$.

\end{proof}

Example \ref{mex} form the previous section is, under $\mathfrak{p} = \cof(\mathcal{M}) = \omega_1$,
in particular under CH, a first countable $\omega D$-bounded (and hence by \ref{wD=wN}
necessarily $\omega N$-bounded) regular -- even locally compact -- separable space that is
not compact. This shows that the assumption $\mathfrak{b} > \omega_1$ cannot simply be
omitted from theorem \ref{b>w_1}. However the following problem remains open.

\begin{problem}
Does $\mathfrak{b} = \omega_1$ imply the existence of a countably tight (or Fr\` echet, or first countable)
regular space that is $\omega D$-bounded but not $\omega$-bounded?
\end{problem}

\section{On products}

It is an immediate consequence of Tychonov's theorem that $\omega$-boundedness is a fully productive
property, while it is also well-known that countable compactness is not productive at all. So, it is
natural to raise the question: how productive are the properties in between these two that we have
considered in this paper. As we shall see, the $\omega D$-bounded property yields an interesting dividing
line for this question. We start with a very simple but useful observation.

\begin{proposition}\label{times}
If the space $X$ is not $\omega$-bounded and the space $Y$ has an infinite discrete subspace
then their product $X \times Y$ is not $\omega D$-bounded.
\end{proposition}

\begin{proof}
Let $A = \{x_n : n < \omega\} \subs X$ be such that $\overline{A}$ is not compact and
$B = \{y_n : n < \omega\} \subs Y$ be infinite and discrete.
Then $$D = \{\langle x_n,y_n \rangle : n < \omega \} \subs X \times Y$$
is discrete (and countable) but cannot have compact closure because $A$ does not.

\end{proof}

Clearly, if $X$ is an $\omega D$-bounded but not $\omega$-bounded space then so is its product
with any finite (discrete) space. The following, perhaps surprising, result says that
an $\omega D$-bounded product  of Hausdorff spaces can fail to be $\omega$-bounded only
in this trivial way.

\begin{theorem}\label{prod-wD}
Assume that the Hausdorff product space $X = \Pi \{X_i : i \in I\}$ is $\omega D$-bounded
but not $\omega$-bounded. Then there is $i \in I$ such that $\Pi \{X_j : j \in I\setm\{i\}\}$ is finite.
\end{theorem}

\begin{proof}
There is $i \in I$ such that $X_i$ is not $\omega$-bounded because the full product
$X$ is not. But then the rest of the product, $Y = \Pi \{X_j : j \in I \setm \{i\}\}$,
cannot be infinite because every infinite Hausdorff space has an infinite discrete subspace,
and then, by proposition \ref{times}, $\,X \cong X_i \times Y$ would not be $\omega D$-bounded,
contradicting our assumption.

\end{proof}

Theorem \ref{prod-wD} says that it is hard for a product of Hausdorff spaces to be $\omega D$-bounded,
except in the trivial case when the factors and the product are $\omega$-bounded. In particular, if $X$ is
Hausdorff and not
$\omega$-bounded then the power $X^\kappa$ is never $\omega D$-bounded if $\kappa > 1$.
Our next result, on the other hand, says that
multiplying by an $\omega D$-bounded space always preserves countable compactness. Thus, when we take a non-trivial
product of finitely many regular $\omega D$-bounded spaces then $\omega D$-boundedness is lost but countable compactness
is preserved.

Actually, this result holds for a natural boundedness property that, at least for Hausdorff spaces, is weaker than
$\omega D$-boundedness and is also implied by sequential compactness.

\begin{definition}
We say that the space $X$ is {\em weakly bounded} iff every infinite subset of $X$ has
an infinite subset of compact closure in $X$.
\end{definition}

\begin{theorem}\label{prod-CC}
Assume that $X$ is a weakly bounded and $Y$ is a countably compact space. Then the product
$X \times Y$ is countably compact as well.
\end{theorem}

\begin{proof}
Let $H \subs X \times Y$ be countably infinite. If the projection $\pi_X(H)$
of $H$ to $X$ is finite then there is a point $x \in \pi_X(H)$ such that
$Z = \{y \in Y : \langle x,y \rangle \in H \}$ is infinite. But then $Z$ has an
accumulation point $z$ in $Y$, and clearly then $\langle x,z \rangle$ is an
accumulation point of $H$ in $X \times Y$.

If, on the other hand, $\pi_X(H)$ is infinite then it has an infinite subset $A$
with compact closure. So then $\overline{A} \times Y$ is countably compact,
moreover $H \cap (A \times Y)$ is infinite and hence has an accumulation point in it.

\end{proof}

Of course, it is an immediate consequence of theorem \ref{prod-CC} that any product of finitely many
weakly bounded spaces is countably compact. Actually, we can prove much more than that.

\begin{theorem}\label{prod-wb}
The product of fewer than $\mathfrak{t}$ weakly bounded spaces is weakly bounded.
\end{theorem}

\begin{proof}
Assume that $\kappa < \mathfrak{t}$ and $\{X_\alpha : \alpha < \kappa\}$ are weakly bounded spaces
and let $A$ be any countably infinite subset of their product $\Pi \{X_\alpha : \alpha < \kappa\}$.
We now define by transfinite recursion sets $A_\alpha \in [A]^\omega$ for  $\alpha < \kappa$
that form a mod finite decreasing sequence, as follows.

Choose first $A_0 \in [A]^\omega$ such that the closure of the projection $\pi_0(A)$ in $X_0$ is compact.
This is possible because either $\pi_0(A)$ is finite, and then $A = A_0$ works, or it's infinite and then
we may use that $X_0$ is weakly bounded. If $0 < \alpha < \kappa$ and the mod finite decreasing sequence
$A_\beta : \beta < \alpha$ has been defined then we first choose a pseudo intersection $B_\alpha$ for it
and then, similarly as in step $0$, we pick $A_\alpha \in [B_\alpha]^\omega$
such that the closure of its projection $\pi_\alpha(A_\alpha)$ in $X_\alpha$ is compact.

After we are done, by $\kappa < \mathfrak{t}$ we can find a pseudo intersection $B \in [A]^\omega$ for the sequence
$\{A_\alpha : \alpha < \kappa\}$. Then for each $\alpha < \kappa$ we have that $B \setm A_\alpha$ is finite,
hence it is clear that $\pi_\alpha(B)$ has compact closure in $X_\alpha$. But this, by Tychonov's theorem, implies that
$B$ has compact closure in the product $\Pi \{X_\alpha : \alpha < \kappa\}$.

\end{proof}

We do not know if theorem \ref{prod-wb} is sharp, hence the following question can be naturally raised.

\begin{problem}
Is the product of $\,\mathfrak{t}$ weakly bounded spaces weakly bounded?
\end{problem}

While this question remains open, we could prove that such a product is at least countably compact.
This result improves theorem 1.6 of \cite{NyV} which says that the product of $\mathfrak{t}$
sequentially compact spaces is countably compact. Moreover, it implies that the
product of $\,\mathfrak{t}$ many $\omega D$-bounded Hausdorff spaces is countably compact.

\begin{theorem}\label{prod-t}
Any product of $\,\mathfrak{t}$ weakly bounded spaces is countably compact.
\end{theorem}

\begin{proof}
Assume that $\{X_\alpha : \alpha < \mathfrak{t}\}$ are weakly bounded spaces
and let $A$ be any countably infinite subset of their product $X = \Pi \{X_\alpha : \alpha < \mathfrak{t}\}$.
We may then define by transfinite recursion sets $A_\alpha \in [A]^\omega$ for all $\alpha < \mathfrak{t}$
that form a mod finite decreasing sequence in exactly the same way as we did in the proof of
theorem \ref{prod-wb}, except that we now do not stop at some $\kappa < \mathfrak{t}$.
Thus we have for each $\alpha < \mathfrak{t}$ that the closure $K_\alpha$ of the
projection $\pi_\alpha(A_\alpha)$ in $X_\alpha$ is compact.

We claim next that $A$ has an accumulation point in the compact subset
$K = \Pi \{K_\alpha : \alpha < \mathfrak{t}\}$ of $X$. Assume, on the contrary, that
this is not the case. Then, by compactness, $K$ has a neighbourhood $W = \cup_{i < n} V_i$,
a finite union of elementary open sets $V_i$ in $X$, such that $W \cap A$ is finite.

For each $i < n$ let $S_i \in [\mathfrak{t}]^{< \omega}$ be the support of $V_i$ and choose
$\alpha < \mathfrak{t}$ such that $S = \cup_{i < n} S_i \subs \alpha$.
Then there is a finite subset $F$ of $A_\alpha$ such that $B = A_\alpha \setm F \subs A_\beta$
for each $\beta \in S$. This implies $\pi_\beta(B) \subs \pi_\beta(A_\beta) \subs \pi_\beta(K_\beta)$
for each $\beta \in S$. But for a point $x \in X$ to belong to $W$ depends only on its
co-ordinates in $S$, hence $K \subs W$ implies $B \subs W$.
Consequently we have $B \subs W \cap A$, contradicting  that $W \cap A$ is finite.
\end{proof}

Again, it is a a natural and interesting question if the statement of theorem \ref{prod-t} is sharp.
This, of course,
ties in with the celebrated Scarborough-Stone problem concerning the countable compactness
of products of sequentially compact spaces, see e.g. \cite{V}.

\end{document}